\documentclass[12pt]{amsart}

\usepackage{amsthm}
\usepackage{amssymb}
\usepackage{amsmath}
\usepackage{graphicx}

\usepackage{color, soul}
\usepackage[colorlinks=true,linkcolor=blue,citecolor=blue]{hyperref}
\usepackage{enumerate}

\usepackage{mathrsfs}

\newtheorem{theorem}{Theorem}[section]
\newtheorem*{theorem*}{Theorem}
\newtheorem{proposition}[theorem]{Proposition}
\newtheorem{corollary}[theorem]{Corollary}
\newtheorem{lemma}[theorem]{Lemma}
\theoremstyle{definition}

\newtheorem{example}[theorem]{Example}

\theoremstyle{remark}

\newtheorem{remark}[theorem]{Remark}

\numberwithin{equation}{section}

\newcommand{\N}{\mathbb{N}}

\newcommand{\C}{\mathbb{C}}
\newcommand{\D}{\mathbb{D}}

\newcommand{\T}{\mathbb{T}}

\def\GP{\mathcal{GP}}
\def\A{\mathcal{A}}
\def\H{\mathcal{H}}
\def\M{\mathcal{M}}

\begin{document}
\baselineskip=.65cm


\title[Gleason parts for algebras of holomorphic functions on $B_{c_0}$]{Gleason parts for algebras of holomorphic functions on the ball of $\mathbf{c_0}$}

\author[R. M. Aron]{Richard M. Aron }
\author[V. Dimant]{Ver\'onica Dimant}
\author[S. Lassalle]{Silvia Lassalle}
\author[M. Maestre]{Manuel Maestre}

\thanks{Partially supported by PAI-UdeSA. The first and fourth authors were partially supported by MINECO and FEDER Project MTM2017-83262-C2-1-P. The second and third
authors were partially supported by Conicet PIP 11220130100483  and ANPCyT PICT 2015-2299. The fourth author was also supported by Project Prometeo/2017/102 of the Generalitat Valenciana. }


\address{Department of Mathematical Sciences, Kent State University, Kent, (OH 44242) USA} \email{aron@math.kent.edu}

\address{Departamento de Matem\'{a}tica y Ciencias, Universidad de San
Andr\'{e}s, Vito Dumas 284, (B1644BID) Victoria, Buenos Aires,
Argentina and CONICET} \email{vero@udesa.edu.ar}

\address{Departamento de Matem\'{a}tica y Ciencias, Universidad de San
Andr\'{e}s, Vito Dumas 284, (B1644BID) Victoria, Buenos Aires,
Argentina and CONICET} \email{slassalle@udesa.edu.ar}

\address{Departamento de An\'{a}lisis Matem\'{a}tico, Universidad de Valencia, Doctor Moliner 50, (46100) Burjasot, Valencia,
Spain} \email{manuel.maestre@uv.es}

\begin{abstract}
For a complex Banach space $X$ with open unit ball $B_X,$  consider the Banach algebras
$\mathcal H^\infty(B_X)$ of bounded scalar-valued holomorphic functions and the subalgebra
$\mathcal A_u(B_X)$ of uniformly continuous functions on $B_X.$ Denoting either algebra
by $\mathcal A,$ we study the Gleason parts of the set of scalar-valued homomorphisms $\mathcal M(\mathcal A)$
on $\mathcal A.$ Following remarks on the general situation, we focus on the case $X = c_0.$

\end{abstract}

\maketitle
\bigskip

\large

\section*{Introduction}

Let $X$ be a complex Banach space with open unit ball $B_X$ and unit sphere $S_X.$ Using standard notation, $\mathcal A_u(B_X)$ denotes the Banach
algebra of holomorphic (complex-analytic) functions $f\colon B_X \to \C$ that are uniformly continuous on $B_X.$ This algebra is clearly
a subalgebra of $\mathcal H^\infty(B_X),$ the Banach algebra of all bounded holomorphic mappings on $B_X$  both endowed with the supremum norm $\|f\|=\sup\{|f(x)|\ |\  \|x\|<1\}$. Also each
function in $\mathcal A_u(B_X)$ extends continuously to $\overline{B}_X$. Then, the  {\em maximal ideal space} (the {\em spectrum} for short) of $\mathcal A_u(B_X),$  that is the set of all nonzero $\C-$valued homomorphisms $\mathcal M(\mathcal A_u(B_X))$
on $\mathcal A_u(B_X),$ contains the point evaluations $\delta_x$ for all $x \in X, \ \|x\| \leq 1.$  Our primary interest here
will be in the structure of the set of such homomorphisms, and our specific focus will be on the Gleason parts of $\mathcal M(\mathcal A_u(B_X))$ and $\mathcal M(\mathcal H^\infty(B_X))$ when $X = c_0.$ Classically, in the case of Banach algebras of holomorphic functions on a finite dimensional space, the study of Gleason parts was motivated by the search for analytic structure in the spectrum. That remains true in our case, in which the holomorphic functions have as their domain the (infinite dimensional) ball of $X.$ However, in  infinite dimensions the situation is more complicated and more interesting. For instance, in this case, we will exhibit non-trivial examples of Gleason parts intersecting more than one fiber; this phenomenon holds in the finite dimensional case in only simple, uninteresting cases.
Unlike the situation when dim $X < \infty,$
it is well-known (see, e.g., \cite{aron-cole-gamelin}) that $\mathcal M(\mathcal A_u(B_X))$ usually contains much more than
mere evaluations at points of $\overline{B}_X$. As we will see, the study of Gleason parts of
$\mathcal M(\mathcal A_u(B_X))$ in the case of an infinite dimensional $X$ is considerably more difficult than in the easy, finite dimensional situation. Now, when the  algebra considered is $\mathcal H^\infty(\D)$ the seminal paper of Hoffman \cite{hoffman} evidences the complicated nature of the Gleason parts for its spectrum (see also \cite{gorkin, suarez, mortini}). So, it is not surprising that our results when $\D$ is replaced by
$B_X$ are incomplete. However, as we will see, much information about Gleason parts for both the $\mathcal A_u$ and $\mathcal H^\infty$ cases can be obtained when $X = c_0.$
\\

As just mentioned,  we will concentrate on the case $X=c_0$, which is the natural extension of the polydisc $\D^n$. After a review in Section 1 of necessary background and some general
results, the description of Gleason parts for  $\mathcal M(\A_u(B_{c_0}))$ will constitute Section 2.  Finally, in Section 3 we will  discuss what we have learned about Gleason parts for $\mathcal M(\mathcal H^\infty(B_{c_0})).$
\\

For general theory of holomorphic functions we refer the reader to the monograph of Dineen~\cite{dineen} and for further information on uniform algebras and Gleason parts we suggest the books of Bear~\cite{bear}, Gamelin~\cite{gamelin}, Garnett~\cite{garnett} and Stout~\cite{stout}.
\vspace{1cm}

\section{Background and general results}

In this section, we will discuss some simple results concerning Gleason parts for $\mathcal M(\mathcal A)$ where $\A$ is an algebra of holomorphic functions defined on the open unit ball of a general Banach space $X$.
Namely, $\mathcal A$ will denote either $\mathcal A_u(B_X)$ or $\mathcal H^\infty(B_X)$.
For a Banach space $X$, as  usual $X^*$ and $X^{**}$ denote the dual and the bidual spaces, respectively. We begin with very short reviews of:\\ (i) Gleason parts (\cite{bear}, \cite{gamelin}) and\\ (ii) the particular Banach algebras of holomorphic functions that we are interested in.
\\

\indent (i)  Let $\mathcal A$ be a uniform algebra  and let
$\mathcal M(\mathcal A)$ denote the compact set of non-trivial homomorphisms $\varphi\colon \mathcal A \to \C$ endowed with the $w(\mathcal A^*,\mathcal A)$ topology ($w^*$ for short).
For $\varphi, \psi \in \mathcal M(\mathcal A),$ we set the {\em pseudo-hyperbolic distance}
$$
\rho(\varphi, \psi) :=
\sup\{|\varphi(f)| \ | \  f \in \mathcal A, \|f\| \leq 1, \psi(f) = 0 \}.
$$
Recall that when $\A = \A(\D)$ or $\A=\H^\infty(\D)$, the pseudo-hyperbolic metric for $\lambda$ and $\mu$ in the unit disc $\D$ is given by
$$
\rho(\delta_\lambda, \delta_\mu) = \Big| \frac{\lambda - \mu}{1 - \overline{\lambda}\mu}\Big|.
$$
Also, the formula given above remains true if $\A=\A(\D)$ for $\lambda, \mu\in \overline{\D}$, if $|\lambda|=1$ and $\lambda \ne \mu$. Clearly, in this case, $\rho(\delta_\lambda, \delta_\mu) =1$.

The following very useful relation is well known (see, for instance, \cite[Theorem~2.8]{bear}):
\begin{equation}\label{Gleason-metric}
\|\varphi - \psi\| = \frac{2 - 2\sqrt{1 - \rho(\varphi, \psi)^2}}{\rho(\varphi,\psi)}.
\end{equation}

Noting that it is always the case that $\|\varphi - \psi\| \ (\equiv \sup_{\|f\| \leq 1}
|\varphi(f) - \psi(f)|) \leq 2,$ the main point here being that $\|\varphi - \psi\| < 2 $
if and only if $\rho(\varphi, \psi) <1.$ From this (with some work), it follows that by defining  $\varphi \sim \psi$
to mean that $\rho(\varphi, \psi) < 1$  leads to a partition of $\mathcal M(\mathcal A)$ into equivalence classes, called {\em Gleason parts}. Specifically,
for each $\varphi \in \mathcal M(\mathcal A),$ the Gleason part containing $\varphi$ is
the set
$$
\mathcal {GP}(\varphi) := \{ \psi \ | \ \rho(\varphi,\psi) < 1\}.
$$
We remark that it was perhaps K\"onig \cite{konig} who coined the phrase {\em Gleason metric} for the metric $\|\varphi - \psi\|.$ \\

\indent (ii) We first recall \cite{davie-gamelin} that any $f \in \H^\infty(B_X)$ can be extended in a canonical way to $\tilde f \in \mathcal H^\infty(B_{X^{\ast\ast}}).$ Moreover,
the extension $f \leadsto \tilde f$ is a homomorphism of Banach algebras. A standard
argument shows that the canonical extension
takes functions in $\mathcal A_u(B_X)$ to functions in $\mathcal A_u(B_{X^{\ast\ast}}).$ Consequently, each point $z_0 \in
B_{X^{\ast\ast}}$ (resp. $\overline{B}_{X^{\ast\ast}}$) gives rise to an element $\tilde{\delta}_{z_0}\in \mathcal M(\mathcal H^\infty (B_X))$ (resp.
$\mathcal M(\mathcal A_u(B_{X}))).$ Here, for a given function $f, \tilde{\delta}_{z_0}(f) =
\tilde{f}(z_0).$ Note that for $f\in \mathcal A_u(B_{X})$ and $z_0\in X^{**}$, with $\|z_0\|=1$,  we are allowed to compute $\tilde f(z_0)$ and we will use this fact without further mention. Also, in order to avoid unwieldy notation we will omit the tilde over the $\delta$, simply writing $\delta_{z_0}(f).$ We recall that either for $\A=\mathcal A_u(B_X)$ or $\A=\mathcal H^\infty(B_X)$ there is a mapping $\pi\colon\mathcal M(\mathcal A) \to \overline{B}_{X^{\ast\ast}}$ given by $\pi(\varphi) := \varphi|_{X^\ast}.$ Note that this makes sense since $X^\ast \subset \mathcal A.$ It is not difficult to see
that $\pi$ is surjective \cite{aron-cole-gamelin}. As usual, for any $z \in \overline{B}_{X^{\ast\ast}},$ the {\em fiber over $z,$}\ will be denoted by
$$
\mathcal M_z := \{ \varphi \in \mathcal M(\mathcal A) \ | \
\pi(\varphi) = z\}.
$$
As we will see, knowledge of the fiber structure is useful in the study of Gleason parts, in the context of the Banach algebras $\mathcal A_u(B_X)$ and $\mathcal H^\infty(B_X).$
The first instance of this occurs in part (b) of Proposition~\ref{basic 1} below.
\vskip .5cm

\begin{proposition}{\label{basic 1}}
Let $X$ be a Banach space and $\mathcal M = \mathcal M(\mathcal A)$ be as above.
\begin{enumerate}[\upshape (a)]
\item The set $\{\delta_z\colon z\in B_{X^{\ast\ast}}\}$ is contained in $\mathcal{GP}(\delta_0)$. In fact, $\rho(\delta_0,\delta_z) = \|z\|$ for each $z \in B_{X^{\ast\ast}}.$
\item Let $z \in S_{X^{\ast\ast}}$ and $w\in B_{X^{\ast\ast}}$. Then, for any $\varphi \in \mathcal M_z$ and $\psi \in \mathcal M_w$,  $\rho(\varphi, \psi) = 1.$ That is, $\varphi$ and $\psi$ lie in different Gleason parts.
\end{enumerate}
\end{proposition}

\begin{proof}
\indent (a) Fix $z \in B_{X^{\ast\ast}}, z \neq 0,$ and $f \in \mathcal A,$ such that $\| f\| \leq 1$ and $f(0) = \delta_0(f) = 0.$
By an application of the Schwarz lemma to $\tilde f \in \mathcal A(X^{\ast\ast}),$ we see that $|\delta_z(f)| = |\tilde{f}(z)|
\leq \|z\|.$ Therefore $\rho(\delta_0, \delta_z) \leq \|z\| < 1,$ or in other words $\delta_z$ is in the same
Gleason part as $\delta_0.$  In addition, if we apply the definition of $\rho$ to a sequence $(x^*_n) \subset \overline{B}_{X^{\ast}} \subset
\mathcal A$ such that $|z(x^*_n)| \to \|z\|,$ we get that $\rho(\delta_0, \delta_z) \geq \|z\|.$

\indent (b) As in part (a) and using that $\varphi\in \M_z$, we may choose a sequence $(x_n^*)$ of norm one functionals on $X$ such that $\varphi(x^*_n) = z(x^*_n) \to \|z\| = 1.$ Observe
that $| \psi(x^*_n) | = |w(x^*_n)| \leq \| w\| < 1.$ For each $n, m \in \N,$ the function $g_{n,m}\colon B_X\to\C$ defined as
$$
g_{n,m}(\cdot) = \frac{(x^*_n(\cdot))^m - w(x^*_n)^m}{\| (x^*_n)^m - w(x^*_n)^m  \|}
$$
is in $\mathcal A = \mathcal A_u(B_X)$ or $\mathcal H^\infty(B_X).$ Evidently, $\| g_{n,m} \| = 1$ and $\psi(g_{n,m}) = 0.$ In addition,
$$
|\varphi(g_{n,m})| \geq \frac{ |z(x^*_n)|^m - \|w\|^m}{1 + \|w\|^m},
$$
which approaches $1$ with $n$ and $m$. Then, $\rho(\psi, \varphi) = 1$ and  $\psi$ and $\varphi$ are in different parts.
\end{proof}

In the classical situation of $\mathcal M(\mathcal H^\infty(\D)),$ the Gleason part containing the evaluation at the origin, $\delta_0,$
consists of the set  $\{ \delta_z \ | \ z \in \D\}.$ This known fact is made evident in view of Proposition \ref{basic 1} and the fact that fibers over points in $\D$ are singletons. In the case of an infinite dimensional space $X$, it can happen that fibers (over interior points) are bigger than single evaluations and also the Gleason part of $\delta_0$ could properly contain $B_{X^{**}}$.
The following, which uses part (a) of Proposition~\ref{basic 1},
gives a glimpse at this situation.\\

\begin{proposition}{\label{gp-at-0}} Let $X$ be a Banach space. Fix $r$, $0 < r < 1$ and consider $B_{X^{\ast\ast}}(0,r) \approx\{ \delta_z \ | \ z \in X^{\ast\ast}, \|z\| < r\} \subset \mathcal M(\mathcal A).$
Then  the closure of $B_{X^{\ast\ast}}(0,r)$ in $\mathcal M(\mathcal A)$ is contained in $\mathcal {GP}(\delta_0).$
\end{proposition}

\begin{proof}
Fix $\varphi \in \mathcal M(\mathcal A)$,  $\varphi $ in the  closure of $B_{X^{\ast\ast}}(0,r)$, and choose any $f \in \mathcal A, f(0)=0, \|f\| = 1.$ By definition, for fixed $\varepsilon>0$ such that $r + \varepsilon <1$ there is $z \in B_{X^{\ast\ast}}(0,r)$ such that $| \varphi(f) - \delta_z(f)| < \varepsilon.$
Then, $$|\varphi(f) - \delta_0(f)| \leq \varepsilon + |\delta_0(f) - \delta_z(f)|\le \varepsilon + \rho(\delta_0, \delta_z) < \varepsilon + r.$$ Thus, $\rho(\varphi,\delta_0) < 1,$ which concludes the proof.
\end{proof}

In many common situations, there are norm-continuous polynomials $P$ acting on the Banach space $X$ whose restriction to $B_X$ is not weakly continuous. To give
one very easy example, the $2-$homogeneous polynomial $P\colon \ell_2 \to \C, \ P(x) = \sum_n x_n^2$ is such that $1 = P(\frac{\sqrt{2}}{2}[e_1 + e_n]) \neq 1/2 = P(\frac{\sqrt{2}}{2}e_1).$
In these cases, the following corollary shows that the exact composition of $\mathcal {GP}(\delta_0)$ is somewhat more complicated.

\begin{corollary}{\label{not-wkly-conts}}
Let $X$ be a Banach space which admits  a (norm) continuous polynomial that is not weakly continuous when restricted to the unit ball. Then $B_{X^{\ast\ast}} \subsetneqq \mathcal {GP}(\delta_0).$
\end{corollary}

\begin{proof} Combining~\cite[Corollary~2]{boyd-ryan} and~\cite[Proposition~3]{boyd-ryan} if $X$ admits a polynomial which is not weakly continuous when restricted to the unit ball, then there is a homogeneous polynomial $P$ on $X$ whose canonical extension $\tilde{P}$ to $X^{\ast\ast}$ is not weak-star continuous at $0$ when restricted to any ball $B_{X^{\ast\ast}}(0,r),\ 0<r<1.$ Fix any $r$ and choose a net $(z_\alpha) \subset B_{X^{\ast\ast}}(0,r)$ that is
weak-star convergent to $0$ and $\tilde{P}(z_\alpha) \nrightarrow 0.$ Choosing a subnet if necessary, we may assume that $\tilde{P}(z_\alpha) \to b \neq 0.$ Applying
Proposition~\ref{gp-at-0}, if  $\varphi \in \mathcal M(\mathcal A)$ is a limit point of $\{\delta_{z_\alpha}\},$ then $\varphi \in \mathcal {GP}(\delta_0).$ Note that
$\delta_0(P) = 0 \neq b = \varphi(P),$ so that $\delta_0 \neq \varphi.$ Finally, $\varphi \in \mathcal M_0,$ since $\pi(\varphi) = \varphi|_{X^\ast},$ which shows
that $\varphi \in \mathcal {GP}(\delta_0) \backslash B_{X^{\ast\ast}}.$
\end{proof}

\begin{remark}
 Note that, under the hypothesis of the above result, by Proposition~\ref{basic 1}, each homomorphism $\varphi\in\mathcal {GP}(\delta_0)\backslash B_{X^{\ast\ast}}$ should be in some fiber over points in $B_{X^{\ast\ast}}$.
\end{remark}

\vspace{.5cm}

In the rest of this section, we will focus on the calculation of the pseudo-hyperbolic distance in some special, albeit important, situations. Here, we will have
to distinguish between the cases $\mathcal A = \mathcal A_u(B_X)$ and $\mathcal A = \mathcal H^\infty(B_X).$

\begin{proposition}{\label{automorphisms}} Let $X$ be a Banach space and
 $\A=\A_u(B_X)$ or $\A=\H^\infty(B_X)$. Suppose that there exists an automorphism $\Phi\colon B_X \to B_X$ and in addition for the case of $\A_u(B_X)$,  assume   $\Phi$ is uniformly continuous. Then, given $x\in B_X$ such that $\Phi(x)=0$,  for any $y\in B_X$ we have
    $$
    \rho(\delta_x,\delta_y) = \|\Phi(y)\|.
    $$
\end{proposition}

\begin{proof}
We only prove the case $\A=\A_u(B_X)$. Let $f \in \mathcal A_u(B_X), \|f\| \leq 1,$ such that $\delta_x(f) = f(x) = 0.$ As $f
\circ \Phi^{-1}$ is in $\mathcal H^\infty(B_X)$, we can apply the Schwarz lemma to obtain
$$
|\delta_y(f)| = |f(y)| = |f \circ \Phi^{-1}(\Phi(y))| \leq \|\Phi(y)\|.
$$
Thus, from the definition of $\rho,$ we see that $\rho(\delta_x,\delta_y) \leq \|\Phi(y)\|.$

\par

For the reverse inequality, choose a norm one functional $x^\ast \in X^\ast$ such that $x^\ast(\Phi(y)) = \|\Phi(y)\|,$
and set
$f = x^\ast \circ \Phi.$ Since $f \in \mathcal A_u(B_X)$ has norm at most $1$ and satisfies $f(x) = 0,$ we get that
$$
\rho(\delta_x, \delta_y) \geq |\delta_y(f)| = \|\Phi(y)\|.
$$
\end{proof}

Note that the proof of Proposition~\ref{automorphisms} shows that $\rho(\delta_x,\delta_y)$ is independent of the particular choice  of the automorphism $\Phi.$

For subsequent embedding results, for a Banach space $X$ and $\A=\A_u(B_X)$ or $\A=\H^\infty(B_X)$ we will use the Gleason metric on $\M(\A)$. As we have already noted  in (i) at the beginning of this section, this metric is the restriction of the usual distance given by the norm on $\A^*$. When we refer to the Gleason metric for elements of $B_{X^{**}}$, the open unit ball $B_{X^{**}}$ will be regarded as a subset of $\M(\A)$. As we will see in the next proposition, under certain conditions, the automorphism $\Phi$ of Proposition \ref{automorphisms} induces  an isometry (for the Gleason metric) in the spectrum that sends some fibers onto different fibers. This type of isometry allows us to transfer information relative to Gleason parts intersecting one fiber to other fibers.  Recall that a \textit{finite type polynomial} on $X$ is a function in the algebra generated by $X^*$. Also, a Banach space $X$ is said to be {\em symmetrically regular} if  every continuous  linear mapping $T\colon X\to X^*$ which is symmetric (i. e. $T(x_1)(x_2)=T(x_2)(x_1)$ for all $x_1, x_2\in X$) turns out to be weakly compact.

\begin{proposition}{\label{Gleason-isometry}}
Let $X$ be a Banach space and $\A=\A_u(B_X)$ or $\A=\H^\infty(B_X)$. Suppose that there exists an automorphism $\Phi\colon B_X \to B_X$ and in addition for the case of $\A_u(B_X)$,  assume   $\Phi$ and $\Phi^{-1}$ are uniformly continuous.
\begin{enumerate}[\upshape (i)]
\item The mapping $\Phi$ induces a composition operator $C_\Phi\colon \A\to \A$, $C_\Phi(f)=f\circ \Phi$ such that  $\Lambda_\Phi:= C_\Phi^t|_{\mathcal{M}(\A)}\colon \mathcal M(\A)\to   \mathcal{M}(\A)$,
the restriction of its transpose to $\M(\A)$,  is an onto isometry for the Gleason metric with inverse $\Lambda_\Phi^{-1}=\Lambda_{\Phi^{-1}}$.

\item If for every $x^*\in X^*$, $x^*\circ\Phi$ and $x^*\circ\Phi^{-1}$ are uniform limits of finite type polynomials then for any $x\in \overline B_{X}$, $\Lambda_\Phi(\M_x)= \M_{\Phi(x)}$. If in addition $X$ is symmetrically regular, then, for any $z\in \overline B_{X^{**}}$, $\Lambda_\Phi(\M_z)= \M_{\widetilde\Phi(z)}$.
\end{enumerate}
\end{proposition}

\begin{proof}
To prove (i), just notice that for $f\in\A$ and $\varphi\in\M(\A)$,
$$
\Lambda_{\Phi^{-1}}(\Lambda_\Phi(\varphi))(f)=\Lambda_\Phi(\varphi)\big(f\circ \Phi^{-1}\big)=\varphi(f).
$$
Through this equality it is easily seen that $\|\Lambda_\Phi(\varphi)-\Lambda_\Phi(\psi)\|=\|\varphi-\psi\|$, for all $\varphi, \psi\in\M(\A)$.

It is enough  to prove (ii) in the case $X$ is symmetrically regular. Fix $z\in \overline B_{X^{**}}$ and take $\varphi\in \M_z$. Given $x_1^*, \ldots, x_n^*$ in $X^*$ as $\varphi$ is multiplicative, we have that
$$ \varphi(x_1^*\cdots x_n^*)=\varphi(x_1^*)\cdots \varphi(x_n^*)=z(x_1^*)\cdots z(x_n^*).$$
Thus, since any polynomial $Q$ of  finite type is a linear combination of elements as above, we have
$$ \varphi(Q)=\widetilde Q(z).$$
 By hypothesis, for any $x^*\in X^*$  there exists a sequence $(Q_k)$ of polynomials of finite type that converges uniformly to $x^*\circ \Phi$ on $B_X$. Hence, the sequence $(\widetilde Q_k)$ converges to $\widetilde x^*\circ \widetilde\Phi$ uniformly on $ B_{X^{**}}$ and $\widetilde \Phi$ admits a unique extension to $\overline B_{X^{**}}$ through weak-star continuity.  Thus,
$$\Lambda_\Phi(\varphi)(x^*)= \varphi(x^*\circ \Phi)= \lim_k \varphi(Q_k)=\lim_k \widetilde Q_k(z)=(\widetilde\Phi (z))(x^*).$$
Consequently, $\Lambda_{\Phi}(\M_z)\subset \M_{\widetilde\Phi(z)}$.  Now, the reverse inclusion follows from (i) because, since $X$ is symmetrically regular and  arguing as in the proof of  \cite[Corollary 2.2]{choi-garcia-kim-maestre},  we know that $\widetilde{\Phi^{-1}}\circ \widetilde\Phi=Id$. Therefore,
$\Lambda_{\Phi}(\M_z)= \M_{\widetilde\Phi(z)}$.
\end{proof}

 To conclude this section, we give three examples of these results.

\begin{example}{\label{ex: c0}}
Let $X = c_0$ and fix a point $x=(x_n) \in B_{c_0}$. Define the mapping $\Phi_x\colon  B_{c_0} \to B_{c_0}$ as follows:
$$
\Phi_x(y) = (\eta_{x_1}(y_1), \eta_{x_2}(y_2),\dots ),
$$
where $\eta_\alpha(\lambda) = \frac{\alpha - \lambda}{1 - \overline{\alpha}\lambda}, \ \alpha, \lambda \in \D.$ In this case $\Phi_x$ is a uniformly continuous automorphism  ($\Phi_x^{-1} = \Phi_x$) with $\Phi_x(x)=0$ and so, for any $y\in B_{c_0}$, $$\rho(\delta_x,\delta_y) = \| \Phi_x(y)\| = \sup_{n \geq 1 }\Big| \frac{x_n - y_n}{1 - \overline{x_n}y_n}\Big| = \sup_{n \geq 1} \rho(\delta_{x_n},\delta_{y_n}).$$\\
Also, $\Lambda_{\Phi_x}$ is an onto isometry for the Gleason metric in $\M(\A)$ both for $\A= \A_u(B_{c_0})$ or $\A=\H^\infty(B_{c_0})$. Moreover, $\Lambda_{\Phi_x}(\M_z)=\M_{\widetilde\Phi_x(z)}$ for any $z\in  \overline B_{\ell_\infty}$.
\end{example}

In the next section, we will discuss the more complicated, more interesting extension of the previous example to $z\in\overline{B}_{\ell_\infty}$; see Theorem~\ref{thm:delta-rho}.

\begin{example}{\label{ex: ell-2-ball}} (\cite[Lemma 4.4]{aron-carando-gamelin-lassalle-maestre})
Let $X = \ell_2$ and fix a point $x \in B_{\ell_2}.$ Define the mapping $\beta_x\colon B_{\ell_2}\to B_{\ell_2}$ as follows:
$$
\beta_x(y) = \frac{1}{1 + \sqrt{1 - \|x\|^2}} \langle\frac{x - y}{1 - \langle y,x\rangle}, x\rangle x + \sqrt{1 - \|x\|^2} \frac{x-y}{1 - \langle y,x\rangle}
$$
$(y \in B_{\ell_2}).$ From~\cite[Proposition 1, p.132]{renaud}, we know that $\beta_x$ is an automorphism from $B_{\ell_2}$ onto itself, with inverse map $\beta_x^{-1} = \beta_x$ and $\beta_x(x)=0$.\\
Also, by expanding $1/[1-\langle y,x\rangle]$ as a geometric series $\sum \langle y,x\rangle ^n$ and noting that the series converges uniformly on $\overline{B}_{\ell_2}$, we see that
$\beta_x(y) = g(y)x+h(y)y$, where the functions $g$ and $h$ are in
$\mathcal A_u(B_{\ell_2})$. Thus, $\beta_x$ is uniformly continuous.
Applying Proposition~\ref{automorphisms}, we see that for all $x, y \in B_{\ell_2}, \
\rho(\delta_x,\delta_y) = \|\beta_x(y)\|.$ Also, by Proposition~\ref{Gleason-isometry}, $\Lambda_{\beta_x}$ is an onto isometry for the Gleason metric in $\M(\A)$, both for $\A= \A_u(B_{\ell_2})$ or $\A=\H^\infty(B_{\ell_2})$. Moreover,  as Proposition~\ref{Gleason-isometry} (ii) holds (see \cite[Lemma 4.3]{aron-carando-gamelin-lassalle-maestre}) $\Lambda_{\beta_x}(\M_y)=\M_{\beta_x(y)}$ for all $y\in\overline B_{\ell_2}$.
\end{example}

\begin{example}{\label{ex:L(H)}}
Let $H$ be an infinite dimensional Hilbert space and let $X = \mathcal L(H)$ be the Banach space of all bounded linear operators from $H$ into itself. Fix $R\in B_{\mathcal L(H)}$ and denote by $R^*$ its adjoint operator. Define the mapping $\Phi_R$ on $B_{\mathcal L(H)}$ as follows:
$$
\Phi_R(T)=(I-RR^*)^{\frac 12}(T - R)(I - R^*T)^{-1}(I-R^*R)^{\frac 12},
$$
($T\in B_{\mathcal L(H)}$). Note that $\Phi_R\colon B_{\mathcal L(H)}\to B_{\mathcal L(H)}$ is an automorphism with inverse map $\Phi_{-R}$ and $\Phi_R(R)=0$. As in the example above, it can be seen that $\Phi_R$ is  uniformly continuous. Then, by Proposition~\ref{automorphisms}, for $R, S\in B_{\mathcal L(H)}$ we obtain $\rho(\delta_R,\delta_S) = \|\Phi_R(S)\|.$ Again, by Proposition~\ref{Gleason-isometry}, $\Lambda_{\Phi_R}$ is an onto isometry for the Gleason metric in $\M(\A)$, both for $\A= \A_u(B_{\mathcal L(H)})$ or $\A=\H^\infty(B_{\mathcal L(H)})$.
\end{example}
\smallskip

\section{Gleason parts for $\mathcal M(\mathcal A_u(B_{c_0})).$}

Compared to other infinite dimensional Banach spaces, what is unusual about $X = c_0$ is that, in relative terms, there are very few continuous polynomials $P\colon c_0 \to \C.$ All
such polynomials are norm limits of finite linear combinations of elements of $c_0^\ast = \ell_1.$ As a consequence, there are very few holomorphic functions on $c_0$ \cite{dineen}. In
particular, every $f \in \mathcal A_u(B_{c_0})$ is a uniform limit of such polynomials.
Thus, since any homomorphism is automatically continuous, its action on $\mathcal A_u(B_{c_0})$ is completely determined by
its action on $c_0^\ast.$ In other words, $\mathcal M(\mathcal A_u(B_{c_0}))$
is precisely $\{ \delta_{z} \ | \ \ z \in \overline{B}_{\ell_\infty} \}.$ Note that if $c_0$ were replaced by $\ell_p,$ this
approximation result would be false, and in fact $\mathcal M(\mathcal A_u(B_{\ell_p}))$ is considerably larger and more complicated than
$\overline B_{\ell_p} \approx \{\delta_z \ | \ z \in \overline B_{\ell_p}\}$ (see, e.g., \cite{farmer}).\\

Our aim here will be to get a reasonably complete description of the Gleason parts of $\mathcal M(\mathcal A_u(B_{c_0})).$ As just mentioned, our work is
greatly helped by the fact that we know exactly what $\mathcal M(\mathcal A_u(B_{c_0}))$ is, namely that it can be associated with
$\overline{B}_{\ell_\infty}.$   A special role is played by homomorphisms $\delta_z$ where $z$ belongs to the distinguished boundary $\T^\N$, the set of all elements $z=(z_n)$ such that $|z_n|=1$ for all $n$. Also, notice that compared with the finite dimensional situation, there is a new and interesting ``wrinkle'' here
in that there are unit vectors $z = (z_n)_n\in\overline B_{\ell_\infty}$ all of whose coordinates have absolute value smaller than $1.$   We begin with
a straightforward lemma.

\begin{lemma}\label{restriction} For any $\varnothing \neq \N_0 \subset \N,$ let $\Gamma\colon\ell_\infty \to \ell_\infty(\N_0)$ be the projection mapping taking $z = (z_j)_{j \in \N} \mapsto
\Gamma(z) = (z_j)_{j \in \N_0}.$ Then for all $z, w \in \overline{B}_{\ell_\infty},$ the following inequality holds:
$$\|\delta_{\Gamma(z)} - \delta_{\Gamma(w)}\| \leq \|\delta_z - \delta_w\|.$$
\end{lemma}

\begin{proof} Clearly, $\Gamma$ is a linear operator having norm 1, and $\Gamma(c_0) = c_0(\N_0).$ Thus each $f \in \mathcal A_u(B_{c_0(\N_0)})$
generates a function $g \in \mathcal A_u(B_{c_0})$ given by $g = f \circ \Gamma|_{c_0}$ having the same norm as $f$. An easy verification shows that
the extension of $g$ to $\mathcal A_u(B_{\ell_\infty})$ is given by $\tilde{g} = \tilde f \circ \Gamma.$ Therefore for all $z, w \in \ell_\infty, \|z\|, \|w\| \leq 1,$

$$\|\delta_{\Gamma(z)} - \delta_{\Gamma(w)}\| = \sup \{| \tilde{f}(\Gamma(z)) - \tilde{f}(\Gamma(w))| \ | \ f \in \mathcal A_u(B_{c_0(\N_0)}), \  \|f\| \leq 1\}$$
$$\leq \sup \{ | \tilde{g}(z) - \tilde{g}(w)| \ | \ g \in \mathcal A_u(B_{c_0}),\ \|g\| \leq 1\} = \| \delta_z - \delta_w\|.  $$
\end{proof}
\bigskip

Another way to restate Lemma~\ref{restriction} is as follows: if $\delta_z \in \mathcal{GP}(\delta_w),$ then $\delta_{\Gamma(z)} \in \mathcal {GP}(\delta_{\Gamma(w)}).$ Since $\N_0$ is allowed to be finite, say of cardinal $k$, if $\delta_z$ and $\delta_w$ are in the same Gleason part, then their projections onto finite coordinates (viewed as being in $\D^k$) are also in the same Gleason part. Our next result examines the  situation: Suppose that $z, w \in \overline{B}_{\ell_{\infty}}$ are such that $\delta_z$ and $\delta_w$ are in the same Gleason part. What can we say about the coordinates where these points differ and where these points are identical?

\begin{lemma}\label{differ}
For $z, w \in \overline{B}_{\ell_\infty}$, let $\N_0 = \{n \in \N \ | \ z_n \neq w_n\}$ and $\Gamma\colon\ell_\infty \to \ell_\infty(\N_0)$ be the projection as in Lemma~\ref{restriction}. Then $$\|\delta_z - \delta_w\|=\|\delta_{\Gamma(z)} - \delta_{\Gamma(w)}\|.$$
\end{lemma}

\begin{proof} Fix $z\in \overline B_{\ell_\infty}$ and define $\Theta_z\colon\ell_\infty(\N_0) \to \ell_\infty$ by:
$$
(\Theta_z(u))_n = \begin{cases} u_n & {\rm \ if \ } n \in \N_0,\\
                            z_n & {\rm \ if \ } n \notin \N_0.
                    \end{cases}
$$
Given $g \in \mathcal A_u(B_{c_0}), \ \|g\| \leq 1,$ let $f = \tilde{g} \circ \Theta_z|_{c_0(\N_0)}.$ Note that $f$ is well-defined since
whenever $u \in \overline{B}_{\ell_\infty(\N_0)}$ then $\Theta_z(u) \in \overline{B}_{\ell_\infty}.$ It is easy to check that $f \in \mathcal A_u(B_{c_0(\N_0)}),$
$\|f\| \leq 1, $ and that $\tilde{f} = \tilde{g} \circ \Theta_z \in \mathcal A_u(B_{\ell_\infty(\N_0)}).$ From the definition of $\N_0,$ we see that

\begin{eqnarray*}
\|\delta_z - \delta_w\| & = & \sup \{ |\tilde{g}(z) - \tilde{g}(w)| \ | \ g \in \mathcal A_u(B_{c_0}), \|g\| \leq 1\}\\
& = & \sup \{ |\tilde{g}(\Theta_z \circ \Gamma(z)) - \tilde{g}(\Theta_z \circ \Gamma(w))| \ | \ g \in \mathcal A_u(B_{c_0}), \|g\| \leq 1\}\\
& \leq & \sup \{ | \tilde{f}(\Gamma(z)) - \tilde{f}(\Gamma(w))| \ | \ f \in \mathcal A_u(B_{c_0(\N_0)}), \|f\| \leq 1\}\\
& = & \|\delta_{\Gamma(z)} - \delta_{\Gamma(w)}\|,
\end{eqnarray*}
and this, with the previous lemma, completes the proof.
\end{proof}

One consequence of this result is that if $z\in \overline{B}_{\ell_\infty}$ with $|z_n| < 1$, for some $n$, then any $w\in \overline{B}_{\ell_\infty}$ such that  $w_j = z_j,$ for all $j \neq n$, and $|w_n| < 1$, satisfies that $\delta_z$ and $\delta_w$ are in the same Gleason part. In particular, the only Gleason parts that are singleton points are the evaluations at points in the distinguished boundary $\T^\N$ of $\overline{B}_{\ell_\infty},$ i.e. the points in the Shilov boundary of $\mathcal M(\mathcal A_u(B_{c_0})).$\\

\begin{lemma}\label{cor-restriction}
For each $n \in \N,$ let $\Gamma_n\colon \ell_\infty \to \ell_{\infty}(\{1,2,\dots,n\})$ be the natural projection. If $z$ and $w$ are both in $\overline{B}_{\ell_\infty},$ then
$$
\|\delta_z - \delta_w\| = \lim_{n \to \infty} \|\delta_{\Gamma_n(z)} - \delta_{\Gamma_n(w)}\| = \sup_{n \in \N} \|\delta_{\Gamma_n(z)} - \delta_{\Gamma_n(w)}\|.
$$
\end{lemma}

\begin{proof}
First, Lemma~\ref{restriction} implies that the sequence $(\|\delta_{\Gamma_n(z)} - \delta_{\Gamma_n(w)}\|)$ is increasing and bounded by $\|\delta_z - \delta_w\|.$ Note also that for each $u \in \overline{B}_{\ell_\infty}, \ \Gamma_n(u) \overset{w(\ell_\infty,\ell_1)}{\longrightarrow} u$, and if $f$ is in $\mathcal A_u(B_{c_0}),$ it follows that $\tilde{f} \in \mathcal A_u(B_{\ell_\infty})$ is weak-star continuous. Consequently,
$\tilde{f}(\Gamma_n(u)) \to \tilde{f}(u)$ as $n \to \infty.$ Therefore, for any $\varepsilon > 0$ take $f \in \mathcal A_u(B_{c_0}), \|f\| \leq 1,$ such that $|\tilde{f}(z) - \tilde{f}(w)|>\|\delta_z -\delta_w\|-\frac{\varepsilon}{2}$. Then,
we can find $n_0 \in \N$ such that both of the following hold:
$$
|\tilde{f}(\Gamma_{n_0}(z)) - \tilde{f}(z)| < \frac{\varepsilon}{4} \quad {\rm  and }\quad  |\tilde{f}(\Gamma_{n_0}(w)) - \tilde{f}(w)| < \frac{\varepsilon}{4}.
$$

Hence, we see that
$$|\tilde{f}(z) - \tilde{f}(w)| \leq \frac{\varepsilon}{4} + |\tilde{f}(\Gamma_{n_0}(z)) - \tilde{f}(\Gamma_{n_0}(w))| + \frac{\varepsilon}{4} \leq \|\delta_{\Gamma_{n_0}(z)} - \delta_{\Gamma_{n_0}(w)}\| + \frac{\varepsilon}{2}.$$

From this, we obtain that  $ \|\delta_z - \delta_w\| \leq \|\delta_{\Gamma_{n_0}(z)} - \delta_{\Gamma_{n_0}(w)}\| + \varepsilon,$ and the lemma follows.
\end{proof}

For the subsequent description of the Gleason parts for $\mathcal M(\mathcal A_u(B_{c_0}))$ we introduce the following notation.
For each $\lambda\in \D$ and $0<r<1$, we denote the {\it pseudo-hyperbolic $r$-disc} centered at $\lambda$ by
$$
\mathcal D_r(\lambda)=\Big\{\mu\in\D\ |\  \rho(\delta_\lambda, \delta_\mu) = \Big| \frac{\lambda - \mu}{1 - \overline{\lambda}\mu}\Big| < r\Big\}.
$$

\begin{theorem}\label{thm:delta-rho}
Let $z = (z_n)$ and $w = (w_n)$ be vectors in $\overline{B}_{\ell_\infty}.$ Then
\begin{equation}\label{(6.1):Cole-Gamelin-Johnson}
\|\delta_z - \delta_w\| = \sup_{n \in \N} \|\delta_{z_n} - \delta_{w_n}\|.
\end{equation}
Moreover, if $\N_0=\{n \in \N \ | \ z_n \neq w_n\}$ then
\begin{equation}\label{ro=sup}
\rho(\delta_z, \delta_w) = \sup_{n \in \N} \rho(\delta_{z_n}, \delta_{w_n})=\sup_{n \in \N_0} \Big|\frac{z_n - w_n}{1 - \overline{z_n}w_n}\Big|.
\end{equation}
Hence, given $z = (z_n)\in\overline{B}_{\ell_\infty}$  we have
$$
\GP(\delta_z)=\bigcup_{0<r<1}\{\delta_w\ |\  w_n=z_n \ \textrm{if}\ |z_n|=1 \ \textrm{and}\  w_n\in\mathcal D_r(z_n)\ \textrm{if}\ |z_n|<1\,\}.
$$\end{theorem}

\begin{proof}  By Lemma~\ref{cor-restriction}, it is enough to see that $\|\delta_{\Gamma_n(z)} - \delta_{\Gamma_n(w)}\|=\sup_{1\le k \le n}\|\delta_{z_k}- \delta_{w_k}\|$ for all $n$, where $\Gamma_n\colon \ell_\infty \to \ell_{\infty}(\{1,2,\dots,n\})$ is the natural projection. By Lemma~\ref{differ}, we may also assume that $z_k\ne w_k$ for $k=1,\ldots, n$.

First, suppose that there exists $k$, $1\le k\le n$, such that $|z_k|=1$ or $|w_k|=1$. Then, $\|\delta_{z_k}- \delta_{w_k}\|=2$ and Lemma~\ref{restriction} gives the equality. Now, assume that $|z_k|, |w_k| <1$ for all $1\le k\le n$. Note that~\eqref{Gleason-metric} describes $\|\delta_{\Gamma_n(z)} - \delta_{\Gamma_n(w)}\|$ in terms of $\rho(\delta_{\Gamma_n(z)}, \delta_{\Gamma_n(w)})$ by an increasing function. Using Example~\ref{ex: c0} we see that $\rho(\delta_{\Gamma_n(z)}, \delta_{\Gamma_n(w)})=\sup_{1\le k \le n} \rho(\delta_{z_k}, \delta_{w_k})$ and both equalities \eqref{(6.1):Cole-Gamelin-Johnson} and \eqref{ro=sup} follow from this.

Now, from $\rho(\delta_z, \delta_w) = \sup_{n \in \N} \rho(\delta_{z_n}, \delta_{w_n})$, we have
$$\GP(\delta_z)=\bigcup_{0<r<1} \{\delta_w\ |\ \rho(\delta_{z_n}, \delta_{w_n})<r, \textrm{ for all } n\}.$$
The conclusion trivially holds.
\end{proof}

Notice that if the algebra is  $\H^\infty(B_{c_0})$ and the vectors $z, w$ belong to the open unit ball $B_{\ell_\infty}$, equation \eqref{(6.1):Cole-Gamelin-Johnson} coincides with equation (6.1) of~\cite[Theorem 6.6]{cole-gamelin-johnson}. The next example illustrates  how Theorem~\ref{thm:delta-rho} can be used.

\begin{example}\label{norm-one}
Consider the following  points in the sphere of $\ell_\infty:$ $z = (1 - \frac{1}{n})_n,
w = (1 - \frac{1}{n^2})_n,$ and $u = (1 - \frac{1}{2n})_n.$  Then $\delta_z$ and $\delta_w$
are in different Gleason parts, while $\delta_z$ and $\delta_u$ are in the same part.\\

\vspace{.1cm}

To see this, observe that
$$
\rho(\delta_z, \delta_w) = \sup_{n \in \N} \rho(\delta_{z_n},\delta_{w_n}) =
\sup_{n \in \N} \left|\frac{z_n - w_n}{1 - \overline{z_n}w_n}\right| = \sup_{n \in \N} \left|\frac{n - n^2}{n^2 + n - 1}\right| = 1,
$$
which shows the first assertion. Similarly,
$$
\rho(\delta_z, \delta_u) =\sup_{n \in \N} \rho(\delta_{z_n},\delta_{u_n}) = \sup_{n \in \N} \left| \frac{z_n - u_n}
{1 - \overline{z_n}u_n}\right| =\sup_{n \in \N} \frac{n}{3n - 1} = \frac{1}{2}.
$$
Thus, $\delta_z$ and $\delta_u$ belong to the same Gleason part.
\end{example}


In order to give a more descriptive insight of the size of the Gleason parts, let us introduce some notation.
Given $z= (z_n)\in\overline B_{\ell_\infty}$, let $\N_1$ be the (possibly empty) set  $\N_1=\{n\in\N\ |\ |z_n|=1\}$. Now, $\N\setminus\N_1$ can be split into two disjoint sets $\N_2\cup \N_3$ such that
$$
\sup_{n\in \N_2} |z_n|<1 \qquad \textrm{ and }\qquad \sup_{n\in \N_3} |z_n|=1.
$$
Note that $\N_2$ and $\N_3$ could be empty and that they are not uniquely determined. For instance, if $\N_3$ is infinite and $\N_2$ is finite, we may redefine $\N_3$ as the union of $\N_3$ and $\N_2$ and redefine $\N_2$  to be empty. Also, $\N_3$ cannot be finite.

In this way we write $\N$ as a disjoint union satisfying the above conditions: $\N=\N_1\cup \N_2\cup \N_3$ and, therefore, the Gleason part containing $\delta_z$ satisfies:
$$
\GP(\delta_z)=\left\{\delta_w\ |\  w_n=z_n\ \text{if}\ n\in\N_1,\ \sup_{n\in\N_2}|w_n|<1 \ \text{and}\ \sup_{n\in\N_3}\left|\frac{z_n-w_n}{1-\overline{z}_nw_n}\right|<1 \,\right\}.
$$
Now, taking into account all the possibilities for the sets $\N_1$, $\N_2$ and $\N_3$ we obtain a more specific description of the different Gleason parts.

\begin{corollary}\label{cor:summary}
Given $z\in\overline B_{\ell_\infty}$ and $\N_1$, $\N_2$, $\N_3$ defined as above, the Gleason part $\GP(\delta_z)$ satisfies one of the following:
\begin{enumerate}[\upshape (i)]

\item If $\N=\N_2$  then $z\in B_{\ell_\infty}$ and $\GP(\delta_z)=\GP(\delta_0)=\{\delta_w\ |\ w \in B_{\ell_\infty} \}$. This produces the identification $\mathcal {GP}(\delta_z) \thickapprox B_{\ell_\infty}.$

\item  If $\N=\N_1$  then $z = (z_n)\in\T^\N$. So,  $\mathcal {GP}(\delta_z)=\{\delta_z\}.$

\item If $\N_3=\varnothing$ and $\N_1,\,\N_2\not=\varnothing$ then $\GP(\delta_z)=\{\delta_w\ |\  w_n=z_n \ \text{if}\ n\in\N_1 \ \text{and}\  \sup_{n\in \N_2} |w_n| < 1\,\}$. So,
    \begin{itemize}
\item    if $\#(\N_2)=k$ then  $\mathcal {GP}(\delta_z) \thickapprox \mathbb D^k$,
\item  if $\N_2$ is infinite, $\mathcal {GP}(\delta_z) \thickapprox B_{\ell_\infty}$.
\end{itemize}
Both identifications are isometries with respect to the Gleason metric.

\item If $\N_3$ is infinite and $\N_2=\varnothing$, then  $\mathcal {GP}(\delta_z)$ contains $\D^k$ for every $k\in\N$ and this inclusion is an isometry for the Gleason metric. There is also a continuous injection of $B_{\ell_\infty}$ into $\mathcal {GP}(\delta_z)$.

\item If both $\N_2$ and $\N_3$ are infinite, then $\mathcal {GP}(\delta_z)$
contains an isometric copy of $B_{\ell_\infty}$, for the Gleason metric.
\end{enumerate}
\end{corollary}

\begin{proof} The results concerning isometries follow from Lemma~\ref{cor-restriction} and Theorem \ref{thm:delta-rho}. We only have to show the continuous injection of $B_{\ell_\infty}$ in item (iv).
If we write $\N_3=\{n_k\}_k$, for each $k$ there exists $r_k>0$ such that whenever $|z_{n_k}-w_{n_k}|<r_k$ we have $w_{n_k}\in\mathbb D$ and
$$
\left|\frac{z_{n_k} - w_{n_k}}{1 - \overline{z_{n_k}}w_{n_k}}\right|<\frac{1}{2}.
$$
Then, denoting $C_{n_k}=r_k\D$ and $C_n=\{0\}$ for $n\not\in \N_3$ we obtain that if $w\in z+ \prod_{n=1}^\infty C_n$ then $\delta_w\in\mathcal {GP}(\delta_z)$. Since it is clear how to inject $B_{\ell_\infty}$ onto the set $z+ \prod_{n=1}^\infty C_n$, we derive the injection of $B_{\ell_\infty}$ into $\mathcal {GP}(\delta_z)$.
\end{proof}
\medskip

\section{Gleason parts for $\mathcal M(\mathcal H^\infty(B_{c_0}))$}

Some of our knowledge about the Gleason parts of $\M(\A_u(B_{X}))$ passes to $\M(\mathcal H^\infty(B_{X}))$ if we consider the restriction mapping $\Upsilon_u\colon \M_{\mathcal H^\infty(B_X)}\to \M_{\A_u(B_X)}$. With obvious notation, it is clear that for any $\varphi, \psi\in \M_{\mathcal H^\infty(B_X)}$,
$$
\rho(\varphi, \psi)\ge \rho_u(\Upsilon_u(\varphi), \Upsilon_u(\psi)).
$$
Therefore, if $\GP_{\A_u}(\Upsilon_u(\varphi))\not=\GP_{\A_u}(\Upsilon_u(\psi))$ we also have $\GP_{\H^\infty}(\varphi)\not= \GP_{\H^\infty}(\psi)$.

\begin{remark} \label{rem:GP in a fiber}
Let $X=c_0$ and consider $z,w\in S_{\ell_\infty}$ such that $\GP_{\A_u}(\delta_z)\not= \GP_{\A_u}(\delta_w)$. Then, for any $\varphi \in \M_z(\H^\infty(B_{c_0}))$ and  $\psi \in \M_w(\H^\infty(B_{c_0}))$, as $\Upsilon_u(\varphi)=\delta_z$ and $\Upsilon_u(\psi)=\delta_w$, we also have $\GP_{\H^\infty}(\varphi)\not= \GP_{\H^\infty}(\psi)$. In particular, if $z\in \overline B_{\ell_\infty}$ belongs to the distinguished boundary $\T^\N$, every $\varphi \in \M_z(\H^\infty(B_{c_0}))$ satisfies $\GP_{\H^\infty}(\varphi)\subset \M_z(\H^\infty(B_{c_0}))$. That is, the Gleason part of $\varphi$ is contained in the fiber over $z$.
\end{remark}

The following is somehow a counterpart to the above remark.

\begin{proposition}\label{reciproca} Let $z,w\in S_{\ell_\infty}$ be such that $\GP_{\A_u}(\delta_z)= \GP_{\A_u}(\delta_w)$.
Then there exist $\varphi \in \M_z(\H^\infty(B_{c_0}))$ and  $\psi \in \M_w(\H^\infty(B_{c_0}))$ satisfying $\GP_{\H^\infty}(\varphi)= \GP_{\H^\infty}(\psi)$.
\end{proposition}

\begin{proof} Fix real numbers $(r_k)$, with $|r_k|<1$ and $r_k\nearrow 1$. Consider the sequences in $B_{\ell_\infty}$:
$$
x^k=r_k z\to z\quad\textrm{ and }\quad y^k=r_k w\to w.
$$
Now, as $\M_z(\H^\infty(B_{c_0}))$ is $w^*$-compact, both $(\delta_{x^k})$ and  $(\delta_{y^k})$
admit $w^*$-convergent subnets $(\delta_{x^{k(\alpha)}})_\alpha$,  $(\delta_{y^{k(\alpha)}})_\alpha$ in $\M(\H^\infty(B_{c_0}))$. Say
$$
\delta_{x^{k(\alpha)}}{\longrightarrow}\varphi;
\qquad \qquad \delta_{y^{k(\alpha)}} {\longrightarrow}\psi.
$$
It is clear that $\varphi \in \M_z(\H^\infty(B_{c_0}))$ and $\psi \in \M_w(\H^\infty(B_{c_0}))$. Now, as $\GP_{\A_u}(\delta_z)= \GP_{\A_u}(\delta_w)$, by Theorem~\ref{thm:delta-rho} we have
$$
C = \sup_n \|\delta_{z_n}-\delta_{w_n}\|_{\M(\A_u(\mathbb D))} =\|\delta_z-\delta_w\|_{\M(\A_u(B_{c_0}))} < 2.
$$

Then, given $f\in \H^\infty(B_{c_0})$,  $\|f\|\le 1$, we can find $\alpha_0$ so that for any $\alpha\ge\alpha_0$,
$$
\left|\delta_{x^{k(\alpha)}}(f)-\varphi(f)\right| <\frac{2-C}{4}\quad \textrm{ and }\quad \left|\delta_{y^{k(\alpha)}}(f)-\psi(f)\right| <\frac{2-C}{4}.
$$

Therefore,
\begin{eqnarray*}
|\varphi(f)-\psi(f)| & \le  & \frac{2-C}{2} + \left|\delta_{x^{k(\alpha)}}(f)-\delta_{y^{k(\alpha)}}(f)\right| \\
& \le  & \frac{2-C}{2} + \left\|\delta_{x^{k(\alpha)}}-\delta_{y^{k(\alpha)}}\right\|_{\M(\H^\infty(B_{c_0}))}\\
& =  & \frac{2-C}{2} + \sup_n \left\|\delta_{x_n^{k(\alpha)}}-\delta_{y_n^{k(\alpha)}}\right\|,
\end{eqnarray*}
where the last equality, which is a version of  the statement of Theorem \ref{thm:delta-rho} for the spectrum  $\M(\H^\infty(B_{c_0}))$, appears in the proof of \cite[Theorem~6.5]{cole-gamelin-johnson}. Now, using the pseudo-hyperbolic distance for the unit disc $\mathbb D$ and  the Schwarz--Pick theorem applied to the function $f(z)=r_{k(\alpha)}z$,  for each fixed $n$ such that $z_n\not= w_n$ we have

\begin{eqnarray*}
\rho(\delta_{x_n^{k(\alpha)}},\delta_{y_n^{k(\alpha)}}) & = & \left| \frac{x_n^{k(\alpha)}-y_n^{k(\alpha)}}{1-\overline{x_n^{k(\alpha)}}y_n^{k(\alpha)}}\right| = \left| \frac{r_{k(\alpha)}(z_n-w_n)}{1-r_{k(\alpha)}^2\overline{z_n}w_n}\right|\\
&\le & \left| \frac{z_n-w_n}{1-\overline{z_n}w_n}\right|\le \rho_u (\delta_z,\delta_w).
\end{eqnarray*}
Then, $\left\|\delta_{x^{k(\alpha)}}-\delta_{y^{k(\alpha)}}\right\|_{\M(\H^\infty(B_{c_0}))} \le \|\delta_z-\delta_w\|_{\M(\A_u(B_{c_0}))} =C$.

Finally, $|\varphi(f)-\psi(f)|\le \frac{2-C}{2} + C= \frac{2+C}{2}$, for any $f\in \H^\infty(B_{c_0})$ with $\|f\|\le 1$. Therefore, $\|\varphi - \psi\|_{\M(\H^\infty(B_{c_0}))}\le \frac{2+C}{2}<2$ and the proof is complete.
\end{proof}

\bigskip

We next prove a kind of extension of the previous proposition. In \cite[Lemma 2.9]{aron-falco-garcia-maestre} it is shown that for $w\in \overline{B}_{\ell_\infty}$ and $b\in\D$ the fibers over $w$ and $(b,w)$ are homeomorphic. To recall the homeomorphism let us consider
$\Lambda_b\colon B_{c_0}\to B_{c_0}$ given by $\Lambda_b(z)=(b,z)$ and let us denote  by $S\colon B_{c_0}\to B_{c_0}$, the shift mapping $S(z)=(z_2, z_3,\dots)$. Now, the homomorphism between the fibers is given by
\begin{eqnarray*}
R_b\colon  \M_w &\to &\M_{(b,w)}\\
\varphi &\mapsto & (f\in\H^\infty(B_{c_0})\mapsto \varphi(f\circ \Lambda_b))
\end{eqnarray*}

Since both $\Lambda_b$ and  $S$ map the unit ball into the unit ball and $S\circ \Lambda_b=Id$ it is easy to see that $R_b$ is an isometry for the Gleason metric. Therefore, the fiber over $w$ and the fiber over $(b,w)$ (for any $w\in \overline{B}_{\ell_\infty}$) intersect the same ``number'' of Gleason parts.

From Remark \ref{rem:GP in a fiber} we know that if $z\in \T^{\mathbb N}$, then every $\varphi \in \M_z(\H^\infty(B_{c_0}))$ satisfies that the Gleason part of $\varphi$ is contained in the fiber over $z$. The next proposition will show us not only that this does not hold for the fibers over points outside $\T^{\mathbb N}$, but also that any Gleason part outside  $\T^{\mathbb N}$ must have elements from different fibers (in fact, at least from a \textit{disc} of fibers).

\begin{proposition} \label{R_b}
Given $b\in\mathbb D$, there exists $r_b>0$ such that if $|c-b|<r_b$ then, for all $\varphi\in\M(\H^\infty(B_{c_0}))$, $R_b(\varphi)$ and $R_c(\varphi)$ are in the same Gleason part.
\end{proposition}

\begin{proof}
By the Cauchy integral formula, $\overline B_{\H^\infty(\mathbb D)}$ is an equicontinuous set of functions. Therefore, there exists $r_b>0$ such that, if $|c-b|<r_b$ then $c\in \mathbb D$ and $|g(b)-g(c)|<1$, for all $g\in B_{\H^\infty(\mathbb D)}$.

Hence, for $f\in \H^\infty(B_{c_0})$ with $\|f\|\le 1$ we have
$$
|f(b,z)-f(c,z)|<1, \qquad\textrm{ if } |c-b|<r_b,\ z\in B_{c_0}.
$$

Therefore, for every $\varphi\in\M(\H^\infty(B_{c_0}))$,

\begin{eqnarray*}
\|R_b(\varphi) - R_c(\varphi)\| & = & \sup_{ \|f\|\le 1} |R_b(\varphi)(f) - R_c(\varphi)(f)|\\
& = & \sup_{ \|f\|\le 1} |\varphi(f\circ\Lambda_b - f\circ\Lambda_c)| \\
&\le & \sup_{ \|f\|\le 1} \|f\circ\Lambda_b - f\circ\Lambda_c \|\\
&=& \sup_{ \|f\|\le 1} \sup_{z\in B_{c_0}} |f(b,z) - f(c,z)| \le 1.
\end{eqnarray*}
\end{proof}
\bigskip

It is clear that the previous result is also valid between the fibers of $w$ and $(w_1, b, w_2,\dots)$ or $(w_1, w_2, b, w_3,\dots)$ and so on. That means that the Gleason part of any morphism in the fiber over a point outside  $\T^{\mathbb N}$, must have elements from other fibers. In particular, there cannot be singleton Gleason parts outside  the fibers over the points in $\T^{\mathbb N}$.

\bigskip

Thus far, the above results show that in $\M(\H^\infty(B_{c_0}))$ there are Gleason parts intersecting different fibers (Propositions~\ref{reciproca} and \ref{R_b}) and there are Gleason parts completely contained in a fiber (Remark~\ref{rem:GP in a fiber}). These results do not provide information on the size of the Gleason parts. In order to understand this feature we appeal to the following result whose statement covers several versions appearing for instance in \cite[Lemma 1.1, p. 393]{garnett}, \cite[Lemma 2.1]{hoffman} and \cite[p. 162]{stout}.

\begin{proposition}
Let $X, Y$ be Banach spaces and $\Omega_X\subset X, \Omega_Y\subset Y$ be open convex subsets. Let $\A$ be a uniform algebra of analytic functions defined  on $\Omega_X$. Suppose that $\Phi\colon\Omega_Y\to \M(\A)$ is an analytic inclusion. Then $\Phi(\Omega_Y)$ is contained in only one Gleason part.
\end{proposition}

\medskip

\begin{remark}
Combining the above proposition with results of \cite{aron-falco-garcia-maestre} and \cite{cole-gamelin-johnson} we derive that \textit{most} of the fibers of $\M(\H^\infty(B_{c_0}))$ contain analytic copies of $B_{\ell_\infty}$ (or $\D$) and each of these copies should be in a single Gleason part. Specifically, we have the following:

\begin{enumerate}[\upshape (i)]
\item By \cite[Theorem 6.7]{cole-gamelin-johnson}, for each $z\in B_{\ell_\infty}$ the fiber over $z$ contains a copy of $B_{\ell_\infty}$. Hence, there is a \textit{thick} intersection of the fiber over $z$ with a Gleason part. This result can be extended to the case of the fibers over $z\in S_{\ell_\infty}$ such that $|z_n|=1$ for $n$ in a finite set $\mathbb N_1$ and $\sup_{n\not\in\mathbb N_1}|z_n|<1$ (see~\cite{dimant-singer}).

\item By \cite[Theorem 2.2]{aron-falco-garcia-maestre}, for each $z\in S_{\ell_\infty}$ with $|z_n|=1$ for all $n$ (or for infinitely many $n$'s \cite{dimant-singer}) the fiber over $z$ contains a copy of $B_{\ell_\infty}$. Hence, there is a \textit{thick} intersection of the fiber over $z$ with a Gleason part.
\item By \cite[Proposition 2.1]{aron-falco-garcia-maestre}, for each $z\in S_{\ell_\infty}$ that attains its norm in $B_{\ell_1}$ the fiber over $z$ contains an analytic copy of the disc $\D$ (which clearly is inside a single Gleason part).
\end{enumerate}
Note that the only case not covered by the previous items corresponds with that of those $z\in S_{\ell_\infty}$ with $|z_n|<1$ for all $n$.
\end{remark}
\medskip

Recall that given a compact set $K$ and  a uniform algebra $\mathcal{A}$ contained in $C(K)$ a point $x\in K$ is called a {\em strong boundary point} for $\mathcal{A}$ if for every neighborhood $V$ of $x$ there exists $f\in \mathcal{A}$ such that $\|f\|=f(x)=1$ and $|f(y)|<1$ if $y\in K\setminus V$.  We see in the next result that in the fiber over each $z\in \T^{\mathbb N}$ there is a strong boundary point. Since the Gleason part of a strong boundary point is just a singleton set, by (ii) of the above remark, we derive that the fiber over any  $z\in \T^{\mathbb N}$ intersects  a thick Gleason part and also a singleton Gleason part.
\medskip

\begin{proposition}
If $\mathcal S$ is the set of strong boundary points of $\M(\H^\infty(B_{c_0}))$ then $\pi(\mathcal S)=\T^{\mathbb N}$.
\end{proposition}

\begin{proof}
Denoting by $\mathcal{SB}$ the  Shilov boundary of $\M(\H^\infty(B_{c_0}))$, we have that $\mathcal S\subset \mathcal{SB}$ (see, e.g., \cite[Corollary 7.24]{stout}) and thus $\pi(\mathcal S)\subset \pi(\mathcal{SB})$. Therefore, in order to prove $\pi(\mathcal S)=\T^{\mathbb N}$ it is enough to see $\pi(\mathcal{SB})\subset\T^{\mathbb N}$ and $\T^{\mathbb N}\subset \pi(\mathcal S)$.

To prove the first inclusion, for each $n\in\mathbb N$, let us consider  the map  $j_n\colon \overline B_{\ell_\infty}\to\overline{\mathbb D}$ given by $j_n(z)=z_n$. Then, $P_n=j_n\circ\pi$ is a weak-star continuous mapping from $\M(\H^\infty(B_{c_0}))$ into $\overline{\mathbb D}$.

Given $a\in\overline B_{\ell_\infty}\setminus\T^{\mathbb N}$, we want to show that $a\not\in\pi(\mathcal{SB})$. Since $a\not\in\T^{\mathbb N}$, there is $n$ such that $|a_n|<1$. The set $C_n=\overline{\mathbb D}\setminus\mathbb D(a_n, \frac{1-|a_n|}{2})$ is a closed subset of $\mathbb C$, so $P_n^{-1}(C_n)$ is weak-star closed in $\M(\H^\infty(B_{c_0}))$. Also, since $C_n$ contains spheres of radius $r$, with $r$ approaching to 1, for each $f\in \H^\infty(B_{c_0})$ we should have
$$
\sup_{z\in B_{c_0}}|f(z)| = \sup_{\varphi\in P_n^{-1}(C_n)}|\varphi(f)|.
$$
Hence,  $P_n^{-1}(C_n)$ is a boundary, which implies that $\mathcal{SB}\subset P_n^{-1}(C_n)$. Thus, $\pi(\mathcal{SB})\subset \pi(P_n^{-1}(C_n))$. Since $a\not\in \pi(P_n^{-1}(C_n))$, we obtain that $a\not\in\pi(\mathcal{SB})$.

For the second inclusion, let $a=(a_n)\in\T^{\mathbb N}$ be given by $a_n=e^{i\theta_n}$, for all $n$. As $(\frac{e^{-i\theta_n}}{2^n})\in \ell_1$ its associated function
$$
x^\ast(x)=\sum_{n=1}^\infty \frac{e^{-i\theta_n}}{2^n}x_n
$$
belongs to $c_0^\ast$. Hence $f(x)=1+x^\ast(x)$ is holomorphic on $c_0$, bounded and uniformly continuous when restricted  to $\overline B_{\ell_\infty}$. Observe that
$$
| f(a)|=2;\qquad\textrm{ while }\qquad | f(z)|<2,\ \textrm{ for all }z\in \overline B_{\ell_\infty}, \ z\not= a.
$$ Associating $f$ with its Gelfand transform $\widehat f$ and noting that $\widehat f$ attains its norm at a strong boundary point \cite[Theorem 7.21]{stout}, there is $\varphi\in \mathcal S$ such that $|\widehat f(\varphi)|=|\varphi(f)|=2$.
Finally
$$
\varphi(f)= \varphi(1)+\varphi(x^\ast)=1+x^\ast(\pi(\varphi))=f(\pi(\varphi)).
$$
Therefore, $\pi(\varphi)=a$, and so $a\in\pi(\mathcal S)$.
\end{proof}
\medskip

Up to now our study about the relationships between fibers and Gleason parts gives information about in which fibers there are singleton Gleason parts, which fibers intersect \textit{thick} Gleason parts and which Gleason parts contain elements of different fibers. To  complete this picture we now wonder about how many Gleason parts intersect a particular fiber. Should it always be more than one?

With respect to this question note that we have already seen that in the fiber over any $z\in\T^\N$ there is a singleton Gleason part and also a copy of $B_{\ell_\infty}$. So, at least two Gleason parts are inside each of these fibers. By translations through mappings $R_b$ (as in Proposition \ref{R_b} and the subsequent comment) we also obtain that there are at least two Gleason parts intersecting the fiber over $z$ for each $z\in S_{\ell_\infty}$ with all but finitely many coordinates of modulus 1.
\vspace{1cm}

The following results show that the fiber over any $z\in B_{\ell_\infty}$ intersects  $2^c$ Gleason parts.
First, relying on the proof of~\cite[Theorem~5.1]{cole-gamelin-johnson} (see also \cite[Corollary~5.2]{cole-gamelin-johnson}) we obtain the desired result for the fiber over $0$. For our purposes, we use the construction and notation given in \cite{cole-gamelin-johnson}.
\medskip

\begin{theorem}\label{theo:betaN GP en M0(Hinf BX)}  Let $X$ be an infinite dimensional Banach space. Then  there is  an embedding  $\Psi\colon (\beta(\N)\setminus \N)\times \D \to \M_0$ that is analytic on each slice $\{\theta\}\times\D$ and satisfies:
\begin{enumerate}[\upshape (a)]
\item $\Psi(\theta,\lambda)\not\in \GP(\delta_0)$ for each $(\theta,\lambda)$.

\item $\GP(\Psi(\theta, \lambda))\cap \GP(\Psi(\tilde \theta, \tilde \lambda))=\varnothing$  for each $\theta, \tilde \theta\in \beta(\N)\setminus \N$ with $\theta\ne \tilde \theta$ and any $\lambda, \tilde \lambda\in \D$.
\end{enumerate}
\end{theorem}

\begin{proof} The existence of the  analytic embedding  $\Psi\colon (\beta(\N)\setminus \N)\times \D \to \M_0$ is given in~\cite[Theorem~5.1]{cole-gamelin-johnson}. Below, we summarize the main ingredients used in its construction.
\begin{itemize}
\item There exists a sequence $(z_k)\subset B_{X^{**}}$  such that $\|z_k\| < \|z_{k+1}\| $ and $\|z_k\|$ is convergent to 1.

\item The sequence of norms $(\|z_k\|)$ increases so rapidly that there exists an increasing sequence $(r_k)$, such that $0<r_k <\|z_k\|$ and $\sum (1-r_k)$ is finite.

\item For a fixed sequence $(a_k)$ so that $0<a_k<1$ and $(a_k)\in \ell_1$, there exists $(L_k)\subset X^*$ such that  $\|L_k\|<1$ and
\begin{itemize}
\item[$\cdot$]  $L_k(z_k)=r_k$,\quad for all $k$,
\item[$\cdot$] $L_j(z_k)=0$, \quad $1<k<j$,
\item[$\cdot$] $|L_j(z_k)| < a_j$,\quad for all $k>j$.
\end{itemize}

\item There exists $0<r<1$ such that for all $k$, if $w_k\colon \D\to X$ is defined as $w_k(\lambda)=\big(\frac{r_k-\lambda}{1-r_k\,\lambda}\big)\frac{z_k}{r_k}$, then $\|w_k(\lambda)\|<1$ for all $|\lambda| < r$.

\item The Blaschke product $G\colon  B_{X^{**}}\to \mathbb C$, given by $G(z)=\prod_{j=1}^\infty \frac{r_j-L_j(z)}{1-r_j\,L_j(z)}$  belongs to $\H^\infty( B_{X^{**}})$ and $|G(z)| <1$ if $\|z\|<1$.

\item For $|\lambda|< r/2$ and each $k$ there exists a unique $\xi_k(\lambda)$ such that $|\xi_k(\lambda)|<r$ and $G(w_k(\xi_k(\lambda)))=\lambda$ for all  $|\lambda|< r/2$.

\item For every $k$ the function $z_k(\lambda)\colon =w_k(\xi_k(\lambda))$ for $|\lambda|<r/2$  is a multiple of $z_k$,  depends analytically on $\lambda$ and satisfies  $\|z_k(\lambda)\|<1$ if $|\lambda|<r/2$ with $z_k(0)=z_k$.
\end{itemize}

Note that replacing $\D$ by $D=\{\lambda \in\C\ |\ \ |\lambda|<r/2 \}$, it is enough to show the result for $\beta(\N)\setminus \N\times D$. The function $\Psi\colon \N\times D\to \M$  defined by $\Psi(k, \lambda)=\delta_{z_k(\lambda)}$ extends to a map  $\Psi\colon \beta(\N)\times D\to \M$ which is continuous on $\beta(\N)$ for each fixed $\lambda$. Moreover, by \cite[Theorem~5.1]{cole-gamelin-johnson}, we know that $\Psi(\beta(\N)\setminus \N\times D)$ lies in the fiber over $0$, $\M_0$.
\bigskip

Now, let us prove that (a) holds. As $\Psi$ is analytic on each slice,  to show that $\Psi(\theta,\lambda)\not\in \GP(\delta_0)$ for each $(\theta,\lambda)$ it is enough to see that $\Psi(\theta, 0)\not\in \GP(\delta_0)$,  for any $\theta$.
Given $N\in\N$, consider $f_N\in \H^\infty( B_{X^{**}})$ defined by
$$
f_N(z)\colon= \prod_{j>N}^\infty \frac{r_j-L_j(z)}{1-r_j\,L_j(z)}.
$$
Then, $\delta_0(f_N)=\prod_{j>N} r_j \to 1$ as $N\to\infty$. On the other hand, as $\Psi(k,0)=\delta_{z_k}$, for $k>N$,
$$
\Psi(k,0)(f_N)=\prod_{j>N}^\infty \frac{r_j-L_j(z_k)}{1-r_j\,L_j(z_k)}=0.
$$
Now, take $\theta\in  \beta(\N)\setminus \N$.  Then, there is a net $(j(\alpha))\subset \mathbb{N}$, such that $\theta=\lim_\alpha j(\alpha)$. Thus,
$$
\Psi(\theta, 0)(f_N)=\lim_\alpha \Psi(j(\alpha),0)(f_N)=0.
$$
Therefore,
$$
\rho(\delta_0, \Psi(\theta, 0)) \ge \sup_N \{|\delta_0(f_N)|\} =\sup_N \{\, \textstyle{\prod}_{j>N} r_j\, \}=1,
$$
which shows that $\Psi(\theta, 0)\not\in \GP(\delta_0)$. \\

To prove (b) let us see that if $\theta\ne \tilde \theta$ then $\GP(\Psi(\theta, D))\cap \GP(\Psi(\tilde \theta, D))=\varnothing$. Indeed, for  $\theta\ne \tilde \theta$ there exists an infinite set $J\subset \N$ such that $\N\setminus J$ is also infinite and  $\theta\in \overline{\{j\colon j\in J\}}$, $\tilde\theta \in \overline{\{j\colon j\in \N\setminus J\}}$.

Here, for $N\in \N$ consider $f_{(J,N)}\in \H^\infty( B_{X^{**}})$ given by
$$
f_{(J,N)}(z)\colon= \prod_{\underset {j> N}{j \in J}}\, \frac{r_j-L_j(z)}{1-r_j\,L_j(z)}.
$$
Then, $\|f_{(J,N)}\|\le 1$ and  $f_{(J,N)}(z_k)=0$ for all $k\in J, k>N$. Hence, as before, we obtain that $\Psi(\theta,0)(f_{(J,N)})=0$.

On the other hand, $\tilde\theta =\lim_{\tilde \alpha} k(\tilde\alpha)$. For these indexes $k(\tilde\alpha)\not\in J$ with $k(\tilde\alpha)>N$, the corresponding factor does not appear in $f_{(J,N)}$ and
$$
\Psi(k(\tilde\alpha),0)(f_{(J,N)})=
\prod_{\underset {N<j<k(\tilde\alpha)}{j \in J}}\frac{r_j-L_j(z_{k(\tilde\alpha)})}{1-r_j\,L_j(z_{k(\tilde\alpha)})} \cdot  \prod_{\underset {j> k(\tilde\alpha)}{j \in J}} r_j.
$$

Notice that $\Big| \frac{r_j-L_j(z_{ k(\tilde\alpha)})}{1-r_jL_j(z_{k(\tilde\alpha)})}\Big| >\frac{r_j-a_j}{1+r_ja_j}$, for $k(\alpha)>j$.
Since  $1 -\frac{r_j-a_j}{1+ r_j\,a_j } < (1-r_j) + 2a_j$, the series
$\sum_{j\ge 1} (1 -\frac{r_j-a_j}{1+ r_j\,a_j})$ converges, implying that  the infinite product $\prod_{j\ge 1} \frac{r_j-a_j}{1+ r_j\,a_j}$ is convergent  as well as the infinite product over $\{j\in J\}$.

Now, given $0<\varepsilon<1$ we can find $k_0\in \mathbb N$ such that for all $k\ge k_0$,
$$
\prod_{\underset{j>k}{j \in J}} r_j>1-\varepsilon \qquad\textrm{ and }\qquad  \prod_{\underset{j>k}{j \in J}} \frac{r_j-a_j}{1+ r_j\,a_j} >1-\varepsilon.
$$

Then,  for $N>k_0$ and $\tilde\alpha$ such that $ k(\tilde\alpha)>k_0$,
we have
$$
\prod_ {\underset{N< j <  k(\tilde\alpha)}{j \in J}} \Big|\frac{r_j-L_j(z_{k(\tilde\alpha)})}{1-r_jL_j(z_{k(\tilde\alpha)})}\Big| >
\prod_ {\underset{N< j <  k(\tilde\alpha)}{j \in J}} \frac{r_j-a_j}{1+r_ja_j}  >  \prod_ {\underset{j>N}{j \in J}} \frac{r_j-a_j}{1+r_ja_j}  > 1 -\varepsilon.
$$
Hence,
$$
|\Psi(k(\tilde\alpha),0)(f_{(J,N)})| > (1 -\varepsilon)^2,
$$
and $|\Psi(\tilde\theta,0)(f_{(J,N)})| \ge (1 -\varepsilon)^2$.
Finally, for any $0<\varepsilon<1$
$$
\rho(\Psi(\theta,0)), \Psi(\tilde\theta,0)) \ge \sup_N \{| \Psi(\tilde\theta,0)(f_{(J,N)})|\}  \ge (1 -\varepsilon)^2,
$$
and the result follows.
\end{proof}

Next, we will see that there is  a bijective biholomorphic mapping from $B_{\ell_\infty}$ into $B_{\ell_\infty}$ which is an isometry for the Gleason metric and transfers each fiber over an interior point to a different fiber. We use this fact to extend the conclusions in Theorem~\ref{theo:betaN GP en M0(Hinf BX)} to the fiber $\M_z(\H^\infty(B_{c_0}))$ for any $z\in B_{\ell_\infty}$.

\begin{lemma}\label{ControlandoMoebius}
Let $\alpha\in\D$ and let $\eta_\alpha\colon \D\to\D$ be the Moebius transformation,
\[
\eta_\alpha(\lambda)=\frac{\alpha-\lambda}{1-\bar{\alpha}\lambda}.
\]
Given $|\alpha|\leq  s<1$, for any  $\lambda\in \D$ with $|\lambda|\leq s$ the following inequality holds:
\[
|\eta_\alpha(\lambda)|\leq \frac{2s}{1+s^2}.
\]
\end{lemma}
\begin{proof} Notice that
$$
1-\Big|\frac{\alpha -\lambda}{1-\bar{\alpha} \lambda}\Big|^2 =\frac{|1-\bar{\alpha}\lambda|^2-|\alpha -\lambda|^2}{|1-\bar{\alpha}\lambda|^2} = \frac{(1-|\lambda|^2)(1-|\alpha|^2)}{|1-\bar{\alpha}\lambda|^2}.
$$
Hence, the result follows for any $|\lambda|\leq s$ since
$$  1- \Big|\frac{\alpha -\lambda}{1-\bar{\alpha}\lambda}\Big|^2\geq  \Big(\frac{1-s^2}{1+s^2}\Big)^2\ \textrm{ and }\  \sqrt{1- \Big(\frac{1-s^2}{1+s^2}\Big)^2}=\frac{2s}{1+s^2}.
$$
\end{proof}
\bigskip

\begin{proposition}\label{Phi_a} Fix $a=(a_n)\in  B_{\ell_\infty}$. The mapping $\Phi_a \colon B_{\ell_\infty} \to B_{\ell_\infty}$, defined by
$$
\Phi_a(z)=(\eta_{a_n}(z_n))
$$
is bijective and biholomorphic. Moreover, for any $x^*\in\ell_1$, the function $x^*\circ \Phi_a$ is uniformly continuous.
\end{proposition}
\begin{proof}
  First, let us check that  $\Phi_a(B_{\ell_\infty})\subset B_{\ell_\infty}$. Fix $z=(z_n)\in B_{\ell_\infty}$ and  take $s=\max\{\|a\|,\|z\|\}  <1$.
Using Lemma~\ref{ControlandoMoebius} we obtain
\[
\|\Phi_a(z)\|=\sup_{n} |\eta_{a_n}(z_n)|\leq \frac{2s}{1+s^2}<1.
\]
To check that $\Phi_a$ is holomorphic, by Dunford's theorem it is enough to check that $\Phi_a$ is weak-star holomorphic, i.e. that $x^*\circ \Phi_a\in \H(B_{\ell_\infty})$ for every $x^*=(b_n)\in \ell_1$. Notice that
$x^*\circ \Phi_a(z)=\sum_{n=1}^{\infty}b_n\eta_{a_n}(z_n)$, and
\[
|b_n \eta_{a_n}(z_n)|\leq |b_n|,
\]
for every $z\in  B_{\ell_\infty}$ and every $n$. By the Weierstrass $M$-test, the series $\sum_{n=1}^{\infty}b_n\eta_{a_n}(z_n)$ converges absolutely and uniformly on $\overline{B}_{\ell_\infty}$ and as each $z\mapsto \eta_{a_n}(z_n)$ belongs to $\A_u(B_{\ell_\infty})$ we have actually proved that $x^*\circ \Phi_a\in  \A_u(B_{\ell_\infty})$, for every $x^*\in \ell_1$. Thus $\Phi_a\in \H(B_{\ell_\infty},B_{\ell_\infty})$.

Finally as $\Phi_a\circ \Phi_{a}(z)=z$ for every $z\in  B_{\ell_\infty}$, we obtain that $\Phi_a$ has inverse $\Phi_a^{-1}=\Phi_{a}$ and $\Phi_a$ is biholomorphic.
\end{proof}
\medskip

\begin{remark}
 Observe that if we consider $a\in B_{c_0}$ and we restrict $\Phi_a$ to $z\in  B_{c_0}$, then we obtain the biholomorphic mapping of Example \ref{ex: c0}.
\end{remark}

Given  $a\in B_{\ell_\infty }$ the  restriction of  $\Phi_a$ to $ B_{c_0}$  will be denoted by ${\Phi_a}\big|_{c_0}$.

\begin{theorem}
Given $a\in B_{\ell_\infty}$, the mapping $C_{\Phi_a}\colon \H^\infty( B_{c_0})\to \H^\infty( B_{c_0})$ defined by
\[C_{\Phi_a}(f)=\tilde{f}\circ {\Phi_a}\big|_{c_0},
\]
where $\tilde{f}\colon B_{\ell_\infty}\to \C$ is the canonical extension of each $f\in \H^\infty(B_{c_0})$, is  an isometric isomorphism of Banach algebras.

Moreover, $\Lambda_{\Phi_a}:= C_{\Phi_a}^t|_{\mathcal{M}(\H^\infty( B_{c_0}))}\colon \mathcal{M}(\H^\infty( B_{c_0}))\to   \mathcal{M}(\H^\infty( B_{c_0}))$,
the restriction of its transpose to $\mathcal{M}(\H^\infty( B_{c_0}))$, is a surjective isometry for the Gleason metric with inverse $\Lambda_{\Phi_a}^{-1}=\Lambda_{\Phi_a}$ that satisfies
\[
\Lambda_{\Phi_a}(\M_z)= \mathcal{M}_{\Phi_a(z)},
\]
for every $z\in B_{\ell_\infty}$.
\end{theorem}

\begin{proof}
Clearly $C_{\Phi_a}$ is well-defined,  $\|C_{\Phi_a}\|\leq 1$ and it is an algebra homomorphism.
Next we claim that given $f\in \H^\infty( B_{c_0})$,
\begin{equation}\label{regularity}
\widetilde{\tilde{f}\circ  {\Phi_a}\big|_{c_0}}=\tilde{f}\circ \Phi_a.
\end{equation}
Let us observe that  $\ell_\infty=C(\beta\N)$  is a symmetrically regular space. Moreover, by Lemma \ref{ControlandoMoebius}, if $0<s<1$, then $m=\sup_{\|z\|\leq s}\|\Phi_a(z)\|<1$. With this in mind, by the method of proof of  \cite[Corollary 2.2]{choi-garcia-kim-maestre}, we have

\begin{equation*}
\widetilde{\tilde{f}\circ  {\Phi_a}\big|_{c_0}}=\tilde{\tilde{f}}\circ \widetilde{ {\Phi_a}\big|_{c_0}}=  \tilde{f}\circ \widetilde{ {\Phi_a}\big|_{c_0}} .
\end{equation*}
By Proposition \ref{Phi_a}, $ {\Phi_a}\big|_{c_0}$ is $w(c_0,\ell_1)$-uniformly continuous on $B_{c_0}$. Hence it has a unique extension to $B_{\ell_\infty}$ that is $w(\ell_\infty,\ell_1)$-uniformly continuous on $B_{\ell_\infty}$  and  it coincides with its canonical extension $\widetilde{ {\Phi_a}\big|_{c_0}}$. On the other hand, also by Proposition \ref{Phi_a}, $\Phi_a$ is $w(\ell_\infty,\ell_1)$-uniformly continuous on $B_{\ell_\infty}$ and it is obviously an extension of $ {\Phi_a}\big|_{c_0}$ to $B_{\ell_\infty}$. Thus, $\widetilde{ {\Phi_a}\big|_{c_0}}(z)=\Phi_a(z)$, for all $z\in B_{\ell_\infty}$.

From this equality we derive that $C_{\Phi_a}\circ C_{\Phi_a}(f)=f$ for every $f \in \H^\infty( B_{c_0})$. Indeed,
$$
C_{\Phi_a}\big( C_{\Phi_a}(f)\big)(z)=\left(\widetilde{\tilde{f}\circ {\Phi_a}\big|_{c_0}}\circ  \Phi_{a}{\big|_{c_0}}\right)(z)={\tilde{f}}\circ \widetilde{ {\Phi_a}\big|_{c_0}} \circ \Phi_a(z)={\tilde{f}}(z)=f(z),
$$
for every $z\in B_{c_0}$.  As a consequence $C_{\Phi_a}$ is an isomorphism of algebras. Also we have $\|f\|\leq \|C_{\Phi_a}\|\| C_{\Phi_a}(f)\|\leq \| C_{\Phi_a}(f)\|$ for every $f$, and therefore $C_{\Phi_a}$ is an isometry.

Hence its transpose $C_{\Phi_a}^t$ when restricted to $\mathcal{M}(\H^\infty( B_{c_0}))$ is well-defined and its range is again in $\mathcal{M}(\H^\infty( B_{c_0}))$.
Moreover,  $\Lambda_{\Phi_a}\circ \Lambda_{\Phi_a}(\varphi)=\varphi$ for every $\varphi \in   \mathcal{M}(\H^\infty( B_{c_0}))$. Finally,   for each $x^*\in \ell_1$, the function $\widetilde{x^*}\circ  {\Phi_a}\big|_{c_0}$ belongs to $\A_u(B_{c_0})$ (as we have already observed) and so it is a uniform limit of finite type polynomials. Hence, as in the proof of Proposition \ref{Gleason-isometry}, we obtain that $\Lambda_{\Phi_a}(\M_z)= \mathcal{M}_{\Phi_a(z)}$,
for every $z\in B_{\ell_\infty}$.
\end{proof}

Combining this last theorem with Theorem \ref{theo:betaN GP en M0(Hinf BX)} we obtain that for each $z\in B_{\ell_\infty}$, the fiber $\M_z(\H^\infty(B_{c_0}))$ contains $2^c$ \textit{discs} lying in different Gleason parts.

\begin{corollary}
Let $z\in B_{\ell_\infty}$. Then,  there is  an embedding of $\Psi\colon (\beta(\N)\setminus \N)\times \D \to \M_z(\H^\infty(B_{c_0}))$ that is analytic on each slice $\{\theta\}\times\D$ and satisfies:
\begin{enumerate}[\upshape (a)]
\item $\Psi(\theta,\lambda)\not\in \GP(\delta_z)$ for each $(\theta,\lambda)$.

\item $\GP(\Psi(\theta, \lambda))\cap \GP(\Psi(\tilde \theta, \tilde \lambda))=\varnothing$  for each $\theta, \tilde \theta\in \beta(\N)\setminus \N$ with $\theta\ne \tilde \theta$ and any $\lambda, \tilde \lambda\in \D$.
\end{enumerate}
\end{corollary}

\textbf{Acknowledgements.} This work was initiated while the first and fourth
authors visited the Departamento de Matem\'atica, Universidad de San Andr\'es during September of 2016. Both of them wish to thank the hospitality they received during their visit.

%

\end{document}